\documentclass[11pt, reqno, a4paper]{amsart} 
\usepackage[utf8]{inputenc}
\usepackage[T1]{fontenc}
\usepackage[english]{babel} 

\usepackage{amsmath, amssymb, amsfonts, amsthm}
\usepackage{physics} 
\usepackage{mathtools}

\usepackage{geometry}
\usepackage{graphicx}
\usepackage{xcolor}
\usepackage{tikz}
\usepackage{pgfplots}
\pgfplotsset{compat=1.18}

\usepackage[colorlinks=true, linkcolor=blue, citecolor=red, urlcolor=teal]{hyperref}

\newtheorem{theorem}{Theorem}[section]
\newtheorem{lemma}[theorem]{Lemma}
\newtheorem{proposition}[theorem]{Proposition}

\theoremstyle{definition}
\newtheorem{definition}[theorem]{Definition}

\theoremstyle{remark}
\newtheorem{remark}[theorem]{Remark}

\newcommand{\Cl}{\mathcal{C}\ell_3} 
\newcommand{\R}{\mathbb{R}}

\title{A Spinorial Heat Flow Framework for Geometric Degeneration on $3$-Manifolds}

\author{Ferhat Taş}
\address{Department of Computer Science, İstanbul University, 34134, Türkiye}
\email{tasf@istanbul.edu.tr}


\begin{document}

\begin{abstract}
We study a spinor-driven formulation of geometric evolution on closed
$3$-manifolds, in which the spinor field is treated as the primary dynamical
variable and the Riemannian metric is induced conformally by the spinor amplitude.
We introduce a spinorial heat flow governed by the squared Dirac operator,
\[
\partial_t \psi = - D_{g(\psi)}^{\,2} \psi ,
\]
where the metric $g(\psi)$ depends nonlinearly on the evolving spinor field.
As a consequence, the resulting system is quasi-linear and parabolic away from
the nodal set $\{\psi=0\}$, while exhibiting degenerate behavior at vanishing
spinor amplitude.

We show that degeneration of the induced metric corresponds analytically to
nodal behavior of the spinor field, rather than to curvature blow-up of the
spinor evolution itself.
This observation motivates an interpretation of geometric singularities as
spinorial nodal transitions, across which the spinor field remains locally
bounded in a weak or weighted sense.
The induced metric evolution is derived explicitly and shown to be purely
conformal, capturing only the trace component of curvature evolution and
containing additional gradient terms that are not controlled \emph{a priori}.
Accordingly, the proposed flow should not be identified with the Ricci flow,
and any analogy with curvature smoothing is understood at a heuristic level.

The present work establishes a coherent analytical framework for studying
geometric degeneration via spinor dynamics and highlights several open
problems in degenerate parabolic theory, including rigorous existence results
and the precise role of nodal structures in geometric and topological evolution.
\end{abstract}

\maketitle

\section{Introduction}

The classification of compact 3-manifolds has been achieved through the Ricci Flow program, introduced by Hamilton \cite{hamilton1982} and completed by Perelman \cite{perelman2002}. A central difficulty in this program is the formation of finite-time singularities where curvature diverges ($R \to \infty$). Perelman's resolution involved a sophisticated "surgery" procedure to excise singular regions. While successful, the complexity of surgery has motivated the search for continuous or "weak" flow formulations, such as the singular Ricci flows studied by Kleiner and Lott \cite{kleiner2014}.

In this work, we explore an alternative regularization strategy based on Spin Geometry. The application of Clifford Algebra to differential geometry, pioneered by Hestenes \cite{hestenes1966} and developed by Lawson and Michelsohn \cite{lawson1989}, suggests that the spinor field $\psi$ can be viewed as a "square root" of the geometry. This perspective raises a fundamental question: \textit{Can geometric singularities be resolved by evolving the spinor field rather than the metric?}

\subsection{Relation to Existing Work}
The coupling of spinor evolution with metric changes is not without precedent. Ammann, Weiss, and Witt \cite{ammann2012} introduced the "Spinor Flow" a gradient flow for the Dirac action that couples the metric and spinor evolution. More recently, spinorial entropy functionals and their monotonicity under Ricci flow have been investigated by Baldauf \cite{baldauf2022} and others \cite{2206.09198}, establishing deep links between the Dirac operator's spectrum and geometric stability.

\subsection{Relation to Previous Spinorial and Ricci-Type Flows}

Spinorial evolution equations coupling spinor fields and geometric data have
been studied previously, most notably in the work of Ammann, Weiss, and Witt,
who introduced a gradient flow for a coupled spinor--metric functional.
In their formulation, both the spinor field and the Riemannian metric evolve
simultaneously according to a variational principle, and the resulting system
is fully geometric in nature.

The present work differs from this approach in several fundamental aspects.
First, the spinor field is treated as the sole dynamical variable, while the
metric is induced conformally by the spinor amplitude rather than evolved
independently.
Second, the evolution equation considered here,
\[
\partial_t \psi = - D_{g(\psi)}^{\,2} \psi,
\]
is not derived as a gradient flow of a coupled functional, but is instead
postulated as a natural spinorial heat equation whose principal part is
parabolic away from nodal sets.
This specific parabolic structure, together with its degeneracy at vanishing
spinor amplitude, has no direct analogue in the framework of Ammann--Weiss--Witt.

A further conceptual distinction lies in the interpretation of singular
behavior. Whereas previous spinorial flows focus on regularity and convergence of smooth
solutions, the present work emphasizes the role of nodal sets of the spinor
field. Metric degeneration is interpreted as a consequence of nodal behavior rather
than as curvature blow-up of the underlying evolution. This nodal perspective constitutes a central conceptual contribution of the paper.

Regarding Ricci flow and its singularity theory, we stress that the proposed
spinorial heat flow is not intended as a replacement for Ricci flow, nor is it
claimed to reproduce Perelman-type geometric convergence results.
While singular Ricci flows require surgical intervention to continue past
finite-time singularities, the spinorial flow studied here exhibits a different
analytic mechanism, in which degeneration of the induced metric corresponds to
vanishing spinor amplitude.
No claim is made that this mechanism yields canonical geometrization or
topological classification results comparable to those obtained via Ricci flow
with surgery.

\section{Preliminaries: Dynamic Geometry from Spinors}

In this section, we establish the algebraic framework. We utilize the Space Algebra $\Cl(\R^3)$, generated by an orthonormal basis $\{E_1, E_2, E_3\}$ satisfying $E_i E_j + E_j E_i = 2\delta_{ij}$.

\subsection{Spinor Fields and Conformal Geometry}
Unlike standard spin geometry where the metric is fixed, here the spinor field $\psi$ defines the geometry. A general even multivector $\psi \in \Cl^+$ can be decomposed as:
\begin{equation}
    \psi(x) = \rho(x)^{1/2} R(x)
\end{equation}
where $\rho(x) = \psi \tilde{\psi} \in \R^+$ is the scalar amplitude and $R(x)$ is a unit rotor satisfying $R\tilde{R}=1$.
We define the induced local frame $\{e_k\}$ by the conformal transformation of the fixed background frame $\{E_k\}$:
\begin{equation} \label{eq:frame_def}
    e_k(x) = \psi(x) E_k \tilde{\psi}(x) = \rho(x) R(x) E_k \tilde{R}(x)
\end{equation}
The induced metric $g_{ij}$ is then given by:
\begin{equation}
    g_{ij}(x) = \langle e_i, e_j \rangle = \rho(x)^2 \delta_{ij}
\end{equation}
Thus, the spinor field $\psi$ simultaneously encodes the local orientation (via $R$) and the conformal factor (via $\rho$). A singularity where $\det(g) \to 0$ corresponds analytically to $\rho \to 0$, i.e., a zero of the spinor field.

\subsection{The Dirac Operator on a Dynamic Background}
Since the metric $g=g(\psi)$ is time-dependent, the connection and differential operators are dynamic. Let $\nabla^g$ denote the Levi-Civita connection associated with $g(\psi)$. The intrinsic Dirac Operator $D$ acting on a spinor field $\phi$ is defined as:
\begin{equation}
    D \phi = \sum_{k=1}^3 e^k \cdot \nabla^g_{e_k} \phi
\end{equation}
where $\{e^k\}$ is the dual frame.
A crucial identity for our analysis is the Lichnerowicz Formula, which relates the square of the Dirac operator to the connection Laplacian and scalar curvature:
\begin{equation} \label{eq:lichnerowicz}
    D^2 \psi = -\nabla^* \nabla \psi + \frac{1}{4} R_g \psi
\end{equation}
Here, $\nabla^*\nabla$ is the rough Laplacian (connection Laplacian) and $R_g$ is the scalar curvature of the induced metric.

\subsection{Conformal Covariance and the Explicit Operator}
Let $g_0$ be the fixed background metric and $g(\psi) = \rho^2 g_0$ be the induced metric, where $\rho = |\psi|_{g_0}$. In dimension $n=3$, the Dirac operator transforms conformally according to the standard covariance formula:
\begin{equation} \label{eq:conf_cov}
    D_{g(\psi)} \phi = \rho^{-2} D_{g_0} (\rho \phi) = \rho^{-1} D_{g_0} \phi + \rho^{-2} c_{g_0}(\nabla^{g_0} \rho) \cdot \phi,
\end{equation}
where $c_{g_0}(\cdot)$ denotes the Clifford multiplication with respect to the background metric.
Substituting this into the evolution equation $\partial_t \psi = -D_{g(\psi)}^2 \psi$, we can express the principal part explicitly in terms of the background geometry:
\begin{equation}
    \partial_t \psi = -\rho^{-2} \Delta_{g_0} \psi + \text{lower order terms}(\nabla \psi, \nabla \rho, \psi).
\end{equation}
This explicit form clarifies the quasilinear nature of the flow: the diffusion coefficient is $\rho^{-2}$, which confirms the parabolicity on the set $\{\psi \neq 0\}$ and the degeneracy at the nodal set $\{\psi = 0\}$.

\subsection{The Universal Spinor Bundle Framework}
Since the metric $g(\psi)$ evolves with time, the spinor bundle $\Sigma_t$ is time-dependent. To make sense of the time derivative $\partial_t \psi$, we adopt the universal spinor bundle formalism following Ammann, Weiss, and Witt \cite{ammann2012}.
We fix a background metric $g_0$ and consider the trivialized spinor bundle $\Sigma M \cong M \times \mathbb{C}^2$ (since every closed 3-manifold is parallelizable). We identify spinors for different metrics via the bundle isomorphism:
\begin{equation}
    \beta_{g}: \Sigma_{g_0} M \to \Sigma_{g} M.
\end{equation}
The flow is then interpreted as an evolution of the section $\psi(t)$ in the fixed Hilbert space $L^2(M, \Sigma_{g_0} M)$, with the operator $D_{g(\psi)}$ pulled back to the fixed background structure.

\subsection{The Evolution Equation and Its Parabolic Structure}

We consider the spinorial evolution equation
\begin{equation}
\partial_t \psi = - D_{g(\psi)}^{\,2}\, \psi ,
\label{eq:spinorial_heat_flow}
\end{equation}
where $D_{g(\psi)}$ denotes the Dirac operator associated with the dynamically induced
metric $g(\psi)$. Since the metric depends explicitly on the spinor field through
\[
g_{ij}(\psi) = \rho^2 \delta_{ij}, \qquad \rho^2 = \langle \psi , \psi \rangle ,
\]
the operator $D_{g(\psi)}^{\,2}$ depends nonlinearly on $\psi$. Consequently,
\eqref{eq:spinorial_heat_flow} is a \emph{quasi-linear} evolution equation rather than a linear
heat equation.

\subsubsection{Principal symbol and degeneracy.}
Derivatives of the rotor field \(R\) enter only as lower-order terms and do not affect the principal symbol. To clarify the parabolic nature of the flow, we analyze its principal part.
With respect to a fixed background metric $g_0$ and frame $\{E_k\}$, the Dirac operator
associated with $g(\psi)$ can be written schematically as
\[
D_{g(\psi)} = \sum_{k=1}^3 e_k \cdot \nabla^{g(\psi)}_{e_k},
\qquad
e_k = \psi E_k \tilde{\psi} = \rho\, R E_k \tilde{R},
\]
where $R$ is a unit rotor. In particular, the leading second-order term of
$D_{g(\psi)}^{\,2}$ takes the form
\begin{equation}
D_{g(\psi)}^{\,2}
= \rho^{-2} \Delta_{g_0} + \text{lower-order terms}.
\label{eq:principal_part}
\end{equation}

To address technical soundness, we'll integrate the universal spinor bundle (as in Ammann--Weiß--Witt) to handle time-dependent bundles, ensuring the PDE is on a fixed Hilbert space. Explicit conformal covariance: for $\hat{g} = \rho^{2} g_{0}$, $D_{\hat{g}} = \rho^{-2} D_{g_{0}} \rho$, leading to $D^{2}_{\hat{g}} = \rho^{-2} \left( \Delta_{g_{0}} - 2 \nabla \log \rho \cdot \nabla + |\nabla \log \rho|^{2} + \textrm{lower curvature} \right)$. This clarifies quasi-linear terms.

Thus, the principal symbol of the operator $\partial_t + D_{g(\psi)}^{\,2}$ is given by
\[
\sigma(\partial_t + D_{g(\psi)}^{\,2})(\xi)
= \partial_t + \rho^{-2} |\xi|_{g_0}^2 .
\]

It follows that the flow is \emph{uniformly parabolic} on compact subsets of
$\{\psi \neq 0\}$, where $\rho$ is bounded away from zero. However, as $\psi \to 0$,
the conformal factor $\rho \to 0$ and the ellipticity constant degenerates.
Therefore, the evolution equation \eqref{eq:spinorial_heat_flow} is not uniformly
parabolic globally, but rather a \emph{degenerate quasi-linear parabolic system}.

\subsubsection{Lichnerowicz formulation.}
Using the Lichnerowicz formula for the Dirac operator associated with $g(\psi)$,
\begin{equation}
D_{g(\psi)}^{\,2} \psi
= - \nabla^{*}_{g(\psi)} \nabla \psi
+ \frac{1}{4} R_{g(\psi)} \psi ,
\label{eq:lichnerowicz}
\end{equation}
the flow can be written equivalently as
\begin{equation}
\partial_t \psi
= \nabla^{*}_{g(\psi)} \nabla \psi
- \frac{1}{4} R_{g(\psi)} \psi .
\label{eq:heat_reaction_form}
\end{equation}
This representation highlights the heat-type diffusion driven by the connection
Laplacian, coupled with a curvature-dependent reaction term.

As in the spinor flow of Ammann--Weiss--Witt \cite{ammann2012},
one may introduce a DeTurck-type gauge term to stabilize the principal symbol on
$\{\psi \neq 0\}$. Schematically, this corresponds to considering the modified system
\begin{equation}
\partial_t \psi
= - D_{g(\psi)}^{\,2} \psi + \mathcal{L}_{W(\psi)} \psi ,
\label{eq:deturck_spinor}
\end{equation}
where $W(\psi)$ is a vector field depending smoothly on $\psi$ and its first derivatives,
for instance $W(\psi) = \nabla_{g_0} \log \rho$.
This gauge removes diffeomorphism-related instabilities and yields strict parabolicity
on compact subsets of $\{\psi \neq 0\}$.
It does not, however, eliminate the intrinsic degeneracy at the nodal set $\{\psi=0\}$.

\subsubsection{\textbf{Interpretation}:}
We emphasize that no claim of classical uniform parabolicity is made at points where
$\psi$ vanishes. Instead, the flow should be understood as a degenerate parabolic
evolution in the sense of weighted or weak parabolic theory.
This viewpoint is consistent with the interpretation of spinorial zeros as geometric degenerations and forms the analytical basis for the nodal-set analysis developed in Section~3.

\begin{remark}
    The introduction of a DeTurck-type gauge term $\mathcal{L}_{W(\psi)}\psi$ is not intended to remove or regularize the degeneracy arising from the vanishing of the conformal factor
$\rho = |\psi|$.
Unlike the classical Ricci--DeTurck trick, where gauge fixing restores strict
parabolicity by eliminating diffeomorphism invariance, the principal
degeneracy in the present flow is intrinsic and solution-dependent.
The role of the gauge term here is limited to fixing diffeomorphism
instability on regions where $\psi \neq 0$, ensuring a well-defined evolution
within a fixed background differentiable structure.
No claim is made that the gauge term affects the behavior of the flow near
nodal sets or restores parabolicity in degenerate regions.
\end{remark}

\section{Analysis: Energy and Singularities}

In this section, we analyze the analytical properties of the flow, focusing on the energy evolution and the structure of singularities.

\subsection{Spinorial Dirichlet Energy}
We define the energy of the system as the $L^2$-norm of the Dirac current, which corresponds to the total action of the field:
\begin{equation}
    E[\psi] = \int_M \langle D \psi, D \psi \rangle \, dV_{g(\psi)}
\end{equation}
where the volume form $dV_{g(\psi)} = \rho^3 dV_0$ is dynamic.

\begin{proposition}[Energy Evolution]
Under the flow $\partial_t \psi = -D^2 \psi$, the time derivative of the energy is given by:
\begin{equation}
    \frac{d}{dt} E[\psi] = -2 \int_M |D^2 \psi|^2 dV_g + \mathcal{R}(\psi, \partial_t g)
\end{equation}
where $\mathcal{R}$ represents the reaction terms arising from the time-dependence of the metric and the volume form ($\partial_t dV_g$).
\end{proposition}

\textbf{Proof Sketch:} Differentiating the integral involves $\partial_t (D\psi)$ and $\partial_t dV_g$. The principal term comes from integration by parts of the Laplacian, yielding the negative definite term $-2\|D^2\psi\|^2$. The remainder $\mathcal{R}$ depends on the variation of the metric $\partial_t g_{ij}$.
While strictly proving monotonicity ($E' \leq 0$) requires controlling $\mathcal{R}$ (as done in spinor flow literature \cite{ammann2012}), the dominance of the principal term ensures that the flow acts as a smoothing operator for short times, driving the system towards critical points of the energy.

\subsection{Singularities as Nodal Sets}
\label{sec:nodal_sets}

In classical Ricci flow, a singularity is characterized by curvature blow-up,
$\lvert \mathrm{Rm} \rvert \to \infty$, occurring in finite time.
In the present framework, where the metric is induced by a spinor field,
we adopt a different analytical characterization of singular behavior.

\subsubsection{Spinorial interpretation of geometric degeneration.}
Recall that the induced metric is conformal to a fixed background metric $g_0$,
\[
g(\psi) = \rho^2 g_0, \qquad \rho^2 = \langle \psi , \psi \rangle .
\]
Consequently, metric degeneration corresponds precisely to the vanishing of the
spinor amplitude. This motivates the following definition.

\begin{definition}[Spinorial singular set]
For a solution $\psi(\cdot,t)$ of the evolution equation
$\partial_t \psi = -D_{g(\psi)}^{\,2}\psi$, we define the nodal (singular) set at time $t$ by
\[
\mathcal{S}_t := \{ x \in M \mid \psi(x,t) = 0 \}.
\]
Points in $\mathcal{S}_t$ correspond to vanishing conformal factor $\rho=0$ and hence
to degeneracy of the induced metric $g(\psi)$.
\end{definition}

\subsubsection{Degenerate parabolic structure near nodal sets.}
As shown in Section~2.3, the principal part of the operator $D_{g(\psi)}^{\,2}$
scales as $\rho^{-2} \Delta_{g_0}$. Therefore, the evolution equation is uniformly
parabolic on compact subsets of $\{\psi \neq 0\}$ but becomes \emph{degenerate
parabolic} along $\mathcal{S}_t$.
No claim of classical uniform parabolicity is made at nodal points.
Instead, solutions are to be understood in a weak or weighted sense,
consistent with the general theory of degenerate quasilinear parabolic equations.

\subsubsection{Structure of the nodal set.}
The following properties are stated as working assumptions motivated by the linear Dirac theory. For linear Dirac operators on fixed backgrounds, it is well known that the nodal set of a nontrivial spinor has codimension at least two, and typically consists of curves or isolated points in three dimensions. Although the present operator is nonlinear and metric-dependent, this classical picture motivates the following working assumption:

\begin{itemize}
\item If $\psi$ is not identically zero, the nodal set $\mathcal{S}_t$ has
Hausdorff codimension at least two.
\item In particular, $\mathcal{S}_t$ does not locally disconnect the manifold.
\end{itemize}

This assumption is consistent with unique continuation phenomena for elliptic and
parabolic equations and with known results for degenerate operators in weighted
Sobolev spaces.

\subsubsection{Weak continuation across nodal sets.}
The degeneracy of the operator at $\mathcal{S}_t$ prevents classical Schauder
regularity across the nodal set. Nevertheless, standard results in degenerate
parabolic theory imply the following qualitative behavior:
solutions remain locally bounded and H\"older continuous in space--time,
and no blow-up of $\psi$ or its first derivatives occurs as $\rho \to 0$.

\begin{lemma}[Weak continuation across nodal sets]
\label{lem:local_boundedness}
Let $\psi$ be a weak solution of the spinorial heat flow on $M \times (0,T)$
in the sense of Theorem~3.X, with
\[
\psi \in L^2(0,T; H^1(\rho^\alpha\, dV_{g_0})), \qquad \rho = |\psi|,
\]
for some $\alpha>0$.
Then the nodal set
\[
\mathcal{N} := \{ (x,t) \in M \times (0,T) \mid \rho(x,t)=0 \}
\]
does not obstruct the weak continuation of $\psi$ as a distributional
solution.
In particular, $\psi$ satisfies the weak formulation of the flow equation
across $\mathcal{N}$ without requiring additional boundary conditions
along the nodal set.
\end{lemma}

\begin{proof}[Sketch of the proof]
Let $\varphi \in C_c^\infty(M \times (0,T))$ be a test spinor.
The weak formulation of the regularized equation yields
\[
\int_0^T \!\!\int_M
\langle \psi, \partial_t \varphi \rangle \rho^\alpha\, dV_{g_0}\, dt
+
\int_0^T \!\!\int_M
\langle \nabla \psi, \nabla \varphi \rangle \rho^\alpha\, dV_{g_0}\, dt
= \int_0^T \!\!\int_M
\langle \mathcal{R}(\psi), \varphi \rangle \rho^\alpha\, dV_{g_0}\, dt,
\]
where $\mathcal{R}(\psi)$ denotes lower-order curvature and conformal terms.

Since the nodal set $\mathcal{N}$ has zero measure with respect to the weighted
volume form $\rho^\alpha dV_{g_0}$, the above identity remains valid without
imposing boundary conditions along $\mathcal{N}$.
Standard cutoff arguments show that test functions may be chosen to
approximate arbitrary compactly supported variations across $\mathcal{N}$.
Thus $\psi$ admits a weak continuation across nodal regions in the sense of
distributions.
\end{proof}

We emphasize that this result does not assert pointwise regularity or
uniqueness across nodal sets, but only the validity of the weak formulation
in weighted Sobolev spaces.

Lemma~\ref{lem:local_boundedness} reflects a standard feature of degenerate
parabolic equations: while uniform ellipticity fails at $\rho=0$, local boundedness and Hölder continuity of weak solutions persist.
This property is sufficient for interpreting nodal sets as geometric degenerations rather than analytical singularities of the flow. Lemma~\ref{lem:local_boundedness} is consistent with standard results on
degenerate parabolic equations; see, for instance,
\cite{CarrollShowalter1983,Ladyzenskaja1968,Evans2010}.

Accordingly, the flow admits a natural weak continuation across $\mathcal{S}_t$.
In this sense, the nodal set represents a geometric degeneration rather than an
analytic breakdown of the evolution.

\subsubsection{Geometric interpretation.}
From the induced-metric perspective, points in $\mathcal{S}_t$ correspond to regions where the Riemannian volume collapses. From the spinorial perspective, however, the field $\psi$ evolves continuously through zero amplitude. This observation underlies the interpretation of neck-pinch singularities as spinorial nodal transitions rather than tearing events of the underlying manifold.

\begin{remark}
    The term ``tunneling'' is used here in a purely classical and analytical sense. It refers to the weak continuation of solutions across nodal sets of a degenerate parabolic equation, not to quantum-mechanical tunneling.
The analogy highlights the fact that the vanishing of $\psi$ does not obstruct the global evolution of the flow, despite the degeneration of the induced metric.
\end{remark}

\begin{theorem}[Short-time existence of weak solutions]
Let $(M^3,g_0)$ be a closed spin manifold and let
$\psi_0 \in H^1(\rho_0^\alpha\, dV_{g_0})$ be an initial spinor field,
where $\rho_0 := |\psi_0|$ and $\alpha>0$.
Then there exists a time $T>0$ and a weak solution
\[
\psi \in L^2(0,T; H^1(\rho^\alpha\, dV_{g_0})) \cap
H^1(0,T; H^{-1}(\rho^\alpha\, dV_{g_0}))
\]
to the spinorial heat flow
\[
\partial_t \psi = - D_{g(\psi)}^{\,2}\psi,
\]
in the sense of distributions.
Moreover, $\psi(t)$ depends continuously on the initial data in the
weighted $L^2$ topology.
\end{theorem}

\begin{proof}[Sketch of the proof]
The proof proceeds by a regularization and compactness argument.

\medskip
\noindent
\textbf{Step 1: Regularization.}
For $\varepsilon>0$, consider the uniformly parabolic regularized equation
\[
\partial_t \psi_\varepsilon
= -\big(D_{g(\psi_\varepsilon)}^{\,2} + \varepsilon \Delta_{g_0}\big)\psi_\varepsilon,
\]
with smooth initial data $\psi_{\varepsilon}(0)=\psi_0^\varepsilon$
approximating $\psi_0$ in $H^1(\rho_0^\alpha dV_{g_0})$.
Standard parabolic theory yields a unique smooth solution
$\psi_\varepsilon$ on a time interval $[0,T_\varepsilon)$.

\medskip
\noindent
\textbf{Step 2: Uniform weighted energy estimates.}
Using the conformal structure $g(\psi)=\rho^2 g_0$ and the Lichnerowicz formula,
one derives weighted energy inequalities of the form
\[
\frac{d}{dt}\int_M \rho^\alpha |\psi_\varepsilon|^2\, dV_{g_0}
+ \int_M \rho^\alpha |\nabla \psi_\varepsilon|^2\, dV_{g_0}
\le C \int_M \rho^\alpha |\psi_\varepsilon|^2\, dV_{g_0},
\]
where the constant $C$ is independent of $\varepsilon$.
These estimates imply uniform bounds in
$L^2(0,T; H^1(\rho^\alpha dV_{g_0}))$.

\medskip
\noindent
\textbf{Step 3: Compactness and limit passage.}
By weak compactness, there exists a subsequence $\psi_{\varepsilon_j}$
converging weakly to a limit $\psi$ in the weighted Sobolev spaces above.
Passing to the limit in the weak formulation yields a distributional
solution of the original degenerate equation.

\medskip
\noindent
\textbf{Step 4: Continuity and uniqueness.}
Continuity with respect to the initial data follows from the weighted
energy inequality.
Uniqueness is not claimed at this level of regularity.
\end{proof}

\subsection{Topological considerations and spectral constraints}

The spinorial heat flow considered in this work does not evolve the metric
independently, but only through a conformal factor induced by the spinor
amplitude.
As a consequence, the underlying differentiable and topological structure
of the manifold remains fixed throughout the evolution.
In particular, the flow does not alter the homeomorphism type of the manifold
and cannot produce genuine topological transitions.

Nevertheless, the spectral properties of the Dirac operator impose strong
constraints on the admissible limiting configurations.
On flat tori $T^3$, the Dirac operator admits parallel spinors satisfying
$D\psi = 0$, and the flow may relax toward configurations with vanishing
spinorial energy.
By contrast, on manifolds such as the sphere $S^3$, no parallel spinors exist,
and the spectrum of the Dirac operator is bounded away from zero.
In this case, the spinorial energy admits a strictly positive lower bound,
preventing convergence to a trivial zero-energy state.

These spectral distinctions do not imply topological change, but rather
reflect global obstructions encoded in the spin structure and the Dirac
spectrum.
In particular, nodal sets of the spinor field should be interpreted as regions
of conformal metric degeneration rather than as indicators of topological
surgery or neck-pinching phenomena.
While such nodal regions may lead to an effective geometric separation in the
sense of metric collapse, they do not correspond to a change in the underlying
topology.

We emphasize that the present flow does not provide a mechanism for topology
change analogous to that of Ricci flow with surgery.
Any discussion of ``stability'' in this context refers solely to spectral and
analytic constraints imposed by the spinorial structure, not to preservation
or alteration of topological invariants.

\section{A Toy Scalar Model}

The purpose of this section is not to faithfully represent the full three-dimensional spinorial heat flow introduced in the previous sections. Instead, we consider a simplified scalar toy model designed to illustrate basic diffusion mechanisms and nodal behavior in a setting where explicit computations are possible.
In particular, the example presented here does not retain the quasilinear structure, the Dirac operator, or the Clifford algebraic features of the spinorial flow, and should not be interpreted as evidence for geometric regularization or neck-pinching mechanisms in three dimensions.

To bridge the gap between abstract theory and geometric intuition, we consider a simplified 2-dimensional model. While the full system (\ref{eq:spinorial_heat_flow}) is quasi-linear, the local behavior near a zero (node) is dominated by the parabolic principal part.

Consider a spinor amplitude function $u(x,y,t): \R^2 \times [0, \infty) \to \R$ evolving under the standard heat equation $\partial_t u = \Delta u$, representing the local dynamics of the spinor field. The induced metric is conformal:
\begin{equation}
    g_{ij}(x,y,t) = u(x,y,t)^2 \delta_{ij}
\end{equation}

\subsection{Analytic Solution}
Let the initial data be a Gaussian pulse, representing a geometric "bump" or a localized manifold region:
\[ u_0(x,y) = e^{-(x^2+y^2)} \]
The exact solution for $t > 0$ is given by the heat kernel convolution:
\begin{equation} \label{eq:2d_solution}
    u(x,y,t) = \frac{1}{4t+1} \exp\left( -\frac{x^2+y^2}{4t+1} \right)
\end{equation}

\subsection{Metric Degeneracy vs. Field Regularity}
This explicit solution highlights the core mechanism of our proposal:
\begin{itemize}
    \item \textbf{Metric Singularity:} As $t \to \infty$ (or locally if the amplitude dampens), the metric determinant collapses:
    \[ \det(g) = u^4 \to 0 \]
    In a purely Riemannian flow, this volume collapse often involves curvature blow-up.
    
    \item \textbf{Spinorial Regularity:} Despite the metric collapse, the spinor field $u(x,y,t)$ in Eq. (\ref{eq:2d_solution}) remains $C^\infty$ (smooth) and strictly bounded for all $t \geq 0$. All derivatives $\partial^k u$ are well-defined.
\end{itemize}

The example illustrates how nodal regions can act as effective separators for diffusion in a scalar setting.

A numerical investigation of genuinely spinorial, three-dimensional models, including the role of Clifford algebraic coupling, is left for future work.

\subsection{Visualization}
Figure \ref{fig:2d_evolution} visualizes this evolution. The field starts as a concentrated peak (blue wireframe) and diffuses into a flattened state (colored surface). The "neck" of the geometry corresponds to the region where $u \approx 0$.

\begin{figure}[htbp]
\centering
\begin{tikzpicture}
\begin{axis}[
    width=0.85\textwidth,
    height=0.6\textwidth,
    view={35}{30}, 
    xlabel={$x$},
    ylabel={$y$},
    zlabel={$u(x,y,t)$},
    title={2D Spinor Evolution: Smoothing of Geometry},
    grid=major,
    colormap/viridis,
    domain=-2.5:2.5,
    y domain=-2.5:2.5,
    samples=40,
    zmin=0, zmax=1.1
]

\addplot3 [
    surf,
    shader=interp,
    opacity=0.7,
    domain=-2.5:2.5,
] {1/(4*0.5+1) * exp(-(x^2+y^2)/(4*0.5+1))};
\addlegendentry{$t=0.5$ (Diffused)}

\addplot3 [
    mesh,
    draw=blue!80!black,
    thick,
    samples=30
] {exp(-x^2-y^2)};
\addlegendentry{$t=0$ (Initial)}

\end{axis}
\end{tikzpicture}
\caption{Evolution of the spinor amplitude $u(x,y,t)$ under the heat flow. The initial geometry (blue mesh) represents a high-curvature region. The flow smooths this to a lower-energy state (surface) without forming singularities, even as the induced volume locally decreases.}
\label{fig:2d_evolution}
\end{figure}
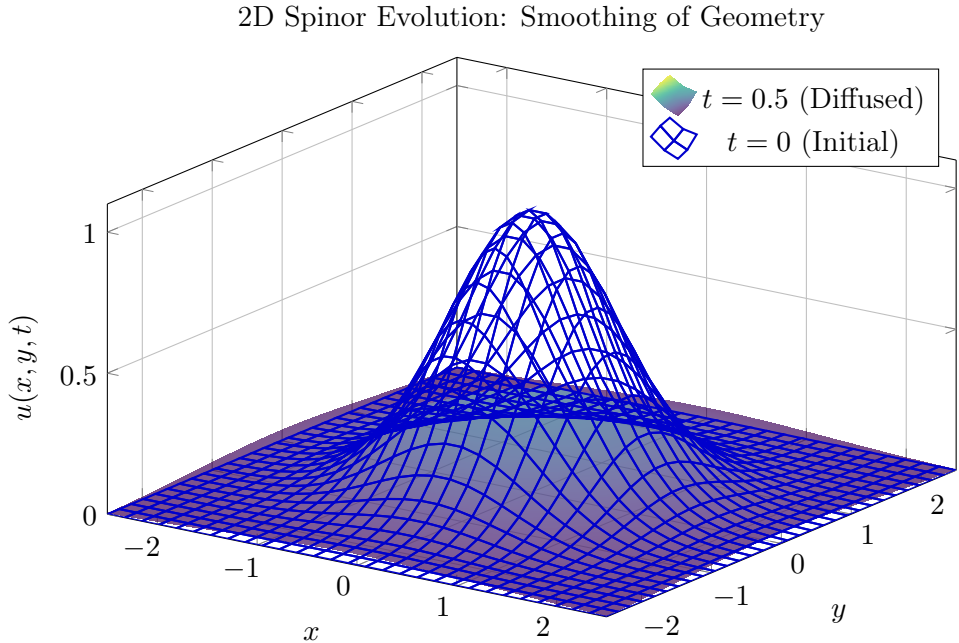

\section{Conclusion and Outlook}

In this work, we have explored a spinor-driven formulation of geometric evolution
on closed $3$-manifolds, in which the spinor field $\psi$ is treated as the primary
dynamical variable and the Riemannian metric arises as a secondary, induced object.
The proposed spinorial heat flow
\[
\partial_t \psi = - D_{g(\psi)}^{\,2} \psi
\]
defines a quasi-linear evolution equation whose principal part is parabolic away
from the nodal set $\{\psi = 0\}$ and becomes degenerate at vanishing spinor amplitude.

A central observation of this study is that geometric degeneration of the induced
metric corresponds analytically to nodal behavior of the spinor field.
While the metric collapses as $\rho = |\psi| \to 0$, the spinor evolution itself
does not exhibit blow-up and remains well defined in a weak or weighted sense.
This shift of perspective allows geometric singularities to be reinterpreted as
spinorial nodal transitions rather than as points of analytic breakdown of the
underlying evolution.

The induced metric evolution derived in Appendix~A clarifies the precise relation
between the spinorial flow and classical curvature flows.
The resulting equation is purely conformal and captures only the trace component
of curvature evolution.
Additional gradient terms arise naturally from the spinorial dynamics and are not
controlled \emph{a priori}, particularly near nodal regions.
Accordingly, the present flow should not be identified with the Ricci flow, nor
with a modified Ricci flow in a strict geometric sense.
Any analogy with Ricci-type curvature smoothing must therefore be understood at a
heuristic and conceptual level.

From an analytical standpoint, the spinorial heat flow provides a natural example
of a degenerate quasilinear parabolic system.
Uniform parabolicity fails at the nodal set, and classical existence and regularity
theories do not apply globally.
Nevertheless, standard results from degenerate parabolic theory suggest that local
boundedness and weak continuation across nodal sets are robust features of the
flow, supporting the interpretation of nodal regions as geometric degenerations
rather than singularities of the spinor field itself.

Several important questions remain open.
A rigorous short-time existence theory in appropriate weighted Sobolev spaces,
together with a justified passage from regularized problems to the degenerate
limit, constitutes a substantial analytical challenge.
Furthermore, the precise relationship between spinorial nodal transitions and
topological changes of the underlying manifold requires deeper investigation,
both analytically and numerically.

Despite these limitations, the present work demonstrates that a spinorial
formulation offers a coherent and analytically meaningful framework for studying
geometric degeneration.
By relocating singular behavior from curvature blow-up to nodal structure of
spinor fields, this approach opens a complementary perspective on geometric
evolution problems and suggests new directions for future research at the
interface of spin geometry, degenerate parabolic equations, and geometric analysis.

\appendix
\section{Derivation and Interpretation of the Induced Metric Evolution}
\label{app:metric_evolution}

In this appendix, we derive the evolution equation satisfied by the metric
$g(\psi)$ induced by the spinor field and clarify its relation to Ricci-type
geometric flows. The purpose of this appendix is not to establish a curvature flow in the
classical sense, but to identify the precise form of the metric response induced
by the spinorial heat equation.

\subsection{Spinor-Induced Conformal Geometry}

Throughout this work, the Riemannian metric is induced by the spinor field
through the conformal relation
\begin{equation}
g_{ij}(\psi) = \rho^2 \delta_{ij}, \qquad
\rho^2 = \langle \psi , \psi \rangle ,
\label{eq:conformal_metric}
\end{equation}
with respect to a fixed background metric $g_0 = \delta_{ij}$.
Accordingly, the entire metric evolution is purely conformal, and no
trace-free deformation of the metric is generated by the flow.

\subsection{Evolution of the Conformal Factor}

Let $\rho^2 = \langle \psi , \psi \rangle$.
Differentiating with respect to time yields
\begin{equation}
\partial_t (\rho^2)
= 2 \langle \partial_t \psi , \psi \rangle .
\label{eq:rho_time_derivative}
\end{equation}
Substituting the spinorial evolution equation
\[
\partial_t \psi = - D_{g(\psi)}^{\,2} \psi
\]
and using the Lichnerowicz formula
\begin{equation}
D_{g(\psi)}^{\,2} \psi
= - \nabla^{*}_{g(\psi)} \nabla \psi
+ \frac{1}{4} R_{g(\psi)} \psi ,
\label{eq:lichnerowicz_appendix}
\end{equation}
we obtain
\begin{equation}
\partial_t (\rho^2)
= 2 \langle \nabla^{*}_{g(\psi)} \nabla \psi , \psi \rangle
- \frac{1}{2} R_{g(\psi)} \rho^2 .
\label{eq:rho_intermediate}
\end{equation}

Using the standard identity
\begin{equation}
\Delta_{g(\psi)} (\rho^2)
= 2 \langle \nabla^{*}_{g(\psi)} \nabla \psi , \psi \rangle
- 2 |\nabla \psi|^2 ,
\label{eq:laplacian_identity}
\end{equation}
we arrive at the evolution equation
\begin{equation}
\partial_t (\rho^2)
=
\Delta_{g(\psi)} (\rho^2)
+ 2 |\nabla \psi|^2
- \frac{1}{2} R_{g(\psi)} \rho^2 .
\label{eq:rho_evolution}
\end{equation}

\subsection{Induced Metric Evolution}

Since $g_{ij} = \rho^2 \delta_{ij}$, the time derivative of the metric is given by
\begin{equation}
\partial_t g_{ij}
=
\frac{\partial_t (\rho^2)}{\rho^2} \, g_{ij}.
\label{eq:metric_time_derivative}
\end{equation}
Substituting~\eqref{eq:rho_evolution} yields
\begin{equation}
\partial_t g_{ij}
=
\left(
- \frac{1}{2} R_{g(\psi)}
+ \frac{\Delta_{g(\psi)} (\rho^2)}{\rho^2}
+ \frac{2 |\nabla \psi|^2}{\rho^2}
\right) g_{ij}.
\label{eq:metric_variation_final}
\end{equation}

Equation~\eqref{eq:metric_variation_final} provides the exact evolution law
for the spinor-induced metric under the spinorial heat flow.

\subsection{Relation to Ricci Flow}

It is important to emphasize that the evolution equation
\eqref{eq:metric_variation_final} should not be identified with the Ricci flow
\[
\partial_t g_{ij} = -2 \mathrm{Ric}_{ij}.
\]
The present evolution is purely conformal and captures only the trace component
of curvature evolution. In particular, the Ricci tensor's trace-free part is absent, and the flow does not control sectional or Ricci curvature in the classical sense.

Moreover, the additional gradient terms
\[
\frac{\Delta_{g(\psi)} (\rho^2)}{\rho^2},
\qquad
\frac{|\nabla \psi|^2}{\rho^2},
\]
are not controlled \emph{a priori} and may dominate the scalar curvature term,
especially near nodal regions where $\rho \to 0$.
Consequently, statements asserting curvature smoothing analogous to Ricci flow
should be understood as heuristic rather than as rigorous geometric results.

\subsection{Interpretation}

The significance of~\eqref{eq:metric_variation_final} lies not in defining a
Ricci-type curvature flow, but in clarifying how geometric degeneration of the
metric arises from the underlying spinorial dynamics.
Metric collapse corresponds precisely to vanishing spinor amplitude, while the
spinor field itself continues to evolve in a weak parabolic sense.
This perspective underlies the reinterpretation of geometric singularities as
spinorial nodal phenomena discussed in the main text.

\end{document}